\documentclass[12pt]{amsart}

\usepackage[text={400pt,660pt},centering]{geometry}

\usepackage{esint}
\usepackage{amssymb,mathrsfs}
\usepackage{graphicx}
\usepackage[colorlinks=true, pdfstartview=FitV, linkcolor=blue, citecolor=blue, urlcolor=blue,pagebackref=false]{hyperref}
\usepackage{xfrac}

\newtheorem{proposition}{Proposition}
\newtheorem{theorem}[proposition]{Theorem}
\newtheorem{lemma}[proposition]{Lemma}
\newtheorem{corollary}[proposition]{Corollary}

\theoremstyle{definition}
\newtheorem{remark}[proposition]{Remark}

\numberwithin{equation}{section}
\numberwithin{proposition}{section}

\newcommand{\N}{\mathbb{N}}
\newcommand{\R}{\mathbb{R}}

\renewcommand{\H}{\mathbb{H}}

\renewcommand{\P}{\mathbb{P}}

\newcommand{\Zd}{\mathbb{Z}^d}
\newcommand{\Rd}{{\mathbb{R}^d}}

\DeclareMathOperator{\range}{ran}

\newcommand{\sym}{\mathrm{sym}}

\newcommand{\indc}{1}

\renewcommand{\a}{\mathbf{a}}

\newcommand{\T}{\mathbf{T}}
\newcommand{\ahom}{\overline{\a}}

\newcommand{\A}{\mathbb{A}}
\newcommand{\Ahom}{\overline{\A}} 

\begin{document}

\title[Unique continuation for periodic equations on large scales]{Optimal unique continuation for periodic elliptic equations on large scales}

\begin{abstract}
We prove a quantitative, large-scale doubling inequality and large-scale three-ellipsoid inequality for solutions of uniformly elliptic equations with periodic coefficients. These estimates are optimal in terms of the minimal length scale on which they are valid, and are at least ``almost'' optimal in the prefactor constants---up to, at most, an iterated logarithm of the initial doubling ratio.
\end{abstract}

\author[S. Armstrong]{Scott Armstrong}
\address[S. Armstrong]{Courant Institute of Mathematical Sciences, New York University, USA}
\email{scotta@cims.nyu.edu}
\author[T. Kuusi]{Tuomo Kuusi}
\address[T. Kuusi]{Department of Mathematics and Statistics, P.O. Box 68 (Gustaf H\"allstr\"omin katu 2), FI-00014 University of Helsinki, Finland}
\email{tuomo.kuusi@helsinki.fi}
\author[C. Smart]{Charles Smart}
\address[C. Smart]{Department of Mathematics, Yale University 10 Hillhouse New Haven, CT 06520}
\email{charles.smart@yale.edu}
\date{July 29, 2021}

\maketitle

\section{Introduction}

A common way of quantifying the unique continuation property for solutions of elliptic PDE is by means of a \emph{doubling inequality}. A typical statement is, that for every solution~$u$ in $B_1$ and constant $M>0$, there exists $C(M) <\infty$ such that 
\begin{align}
\label{e.harmonic.doubling}
\left\| u \right\|_{L^2(B_1)} 
\leq 
M \left\| u \right\|_{L^2(B_{\sfrac12})} 
\quad \implies \quad 
\sup_{r\in (0,1)} 
\frac{\left\| u \right\|_{L^2(B_r)} }{ \left\| u \right\|_{L^2(B_{\sfrac r2})}}
\leq 
C(M).
\end{align}
Related to doubling inequalities is a \emph{three-ball inequality}, which states that a solution must satisfy
\begin{align}
\label{e.free.ballz}
\left\| u \right\|_{L^2(B_{\theta r} )} 
\leq 
C
\left\| u \right\|_{L^2(B_{\theta^2 r})}^{\alpha} 
\left\| u \right\|_{L^2(B_{\theta r} )}^{1-\alpha} 
\,, \quad \forall \theta,r \in (0,1)\,.
\end{align}
for $C<\infty$ and exponent $\alpha\leq \frac12$, ideally for $\alpha=\frac12$ and $C=1$. 
Harmonic functions satisfy these two estimates, with $C(M) = M$ in the doubling inequality and $(\alpha,C)=(\sfrac 12,1)$ in the three-ball inequality. In this case, the two are actually equivalent and can be obtained from the fact that
\begin{equation}
\label{e.threespheres}
\mbox{the mapping} 
\quad 
t \mapsto \log \left( \fint_{\partial B_{\exp(t)}} u^2 \right) 
\quad 
\mbox{is convex.} 
\end{equation}
This can be checked by a direct computation. They can also be proved from the Almgren frequency formula, which is essentially the same as~\eqref{e.threespheres}, or proved directly from the $L^2$-orthogonality of homogeneous harmonic polynomials of different degrees on spheres centered at the origin. 

\smallskip

Both of these inequalities obviously imply, on a qualitative level, the unique continuation property: a harmonic function cannot vanish in a ball without vanishing identically. They also contain quantitative information regarding the growth of harmonic functions which have important applications, including to spectral properties of the operator, the behavior of nodal sets of eigenfunctions, and inverse problems (see for instance~\cite{L,Ken,LM,ARRV}). 

\smallskip

Our focus here is on doubling and three-ball inequalities for elliptic operators with periodic coefficients. 
Such operators may behave like any variable-coefficient elliptic operators on \emph{small} length scales (below the underlying period of coefficients), but by the theory of elliptic homogenization they must behave like a constant-coefficient operator on \emph{large} scales. Therefore, one may expect to obtain analogues of~\eqref{e.harmonic.doubling} and~\eqref{e.threespheres} for solutions periodic operators, at least on large enough length scales. 
Such results have indeed been proved recently by various authors~\cite{LS,KZ,AKS,KZZ,Z}. 
However, the estimates in these works have been either qualitative (the dependence of various constants in the estimates is implicit) or else quantitative, but sub-optimal, in terms of the range of length scales on which they are valid, or else in terms of estimates for the constant~$C(M)$ in~\eqref{e.harmonic.doubling} or the parameters~$(\alpha,C)$ in~\eqref{e.free.ballz}.

\smallskip

In this paper, we prove \emph{optimal} quantitative three-ball and doubling inequalities for periodic elliptic equations on large scales. We provide a sharp estimate of the minimal length scale on which the results are valid as well as displaying estimates with prefactor constants which are either optimal, or else very close to optimal, in their dependence on the initial doubling ratio~$M$. 

\subsection*{Statement of the main result}
We consider the linear, divergence-form, uniformly elliptic equation
\begin{equation}
\label{e.pde}
-\nabla\cdot \a \nabla u = 0 \quad \mbox{in} \ U \subseteq\Rd,
\end{equation}
in dimension $d\geq 2$, where the coefficient field~$\a(\cdot)$ is a measurable, $\Zd$--periodic map from $\Rd$ into the set~$\R^{d\times d}_{\sym}$ of~$d$-by-$d$ real symmetric matrices satisfying a uniform ellipticity condition with ellipticity constant~$\Lambda$. That is, we assume that, for a given~$\Lambda \in [1,\infty)$,
\begin{equation}
\label{e.ellipticity}
I_d \leq \a(x) \leq \Lambda I_d, \quad \forall x\in\Rd\,,
\end{equation}
and
\begin{equation}
\label{e.periodicity}
\a(\cdot + z) = \a, \quad \forall z\in\Zd. 
\end{equation}
In particular, since we are interested in large scales, we make no assumption of regularity on the coefficient field~$\a(\cdot)$ beyond measurability.

\smallskip

We let~$\ahom$ be the homogenized matrix arising in the  periodic homogenization of~\eqref{e.pde}. We note that~$\ahom$ is a symmetric matrix and satisfies~$I_d \leq \ahom \leq \Lambda I_d$. It is natural to work with the ellipsoids for the operator~$\nabla \cdot \ahom \nabla$, rather than balls. These are denoted by
\begin{align}
\label{e.Er}
E_r := \left\{ x \in \Rd \,:\, x \cdot \ahom^{\,-1} x \leq r^2 \right\}
\,.
\end{align}
We use the notation $r\wedge s := \min\{ s,r\}$. 

\smallskip

We are now ready to state the main result of the paper. 

\begin{theorem}[{Large-scale three-ellipsoid and doubling inequalities}]
\label{t.threeball}
\emph{} \\
There exist constants
$C(d,\Lambda)\in [1,\infty)$ and~$\theta(d,\Lambda)\in \bigl(0,\tfrac 12\bigr]$ such that, for every $M \in [5,\infty)$, $R \in [ C \log M,\infty)$ and $u\in H^1(E_{2R})$ satisfying
\begin{align*}
-\nabla \cdot \a\nabla u = 0 \quad \mbox{in} \ E_{2R}, 
\end{align*}
and
\begin{align}
\label{e.initial.doubling}
\left\| u \right\|_{L^2(E_{R})} 
\leq 
M
\left\| u \right\|_{L^2(E_{\theta R})} 
\,,
\end{align}
we have, for every~$r \in \left[ C \log M , R\right]$,
\begin{align}
\label{e.threeball.logM}
\left\| u \right\|_{L^2(E_{\theta r})} 
\leq 
M^{C (\log \log M) \wedge C\log \frac Rr}
\left\| u \right\|_{L^2(E_{\theta^2 r})}^{\sfrac12} 
\left\| u \right\|_{L^2(E_{r})}^{\sfrac12}
\,,
\end{align}
and
\begin{align}
\label{e.doubling.logM}
\frac{\left\| u \right\|_{{L}^2(E_{r})}}{ \left\| u \right\|_{{L}^2(E_{\theta r})} }
\leq
M^{C (\log \log M) \wedge C\log \frac Rr}
\,.
\end{align}
\end{theorem}

The main point of the above theorem is that the three-ellipsoid and doubling inequalities are valid all the way down to scale~$r = C\log M$.
This requirement that the length scale be larger than this \emph{minimal scale} of~$C\log M$ is essentially a growth condition on the solution: the theorem applies to solutions which behave like a slow exponential on the largest ball. This is a natural condition, since slow exponential functions are the fastest growing solutions which can be approximated using the usual asymptotic corrector expansions in the theory of periodic homogenization, and therefore we expect that this is the best possible result which can be proved from standard homogenization techniques. 
Indeed, our estimate of $C\log M$ for the minimal scale is optimal in all dimensions~$d>2$, at least for coefficients~$\a(\cdot)$ which are merely~$C^{0,\alpha}$ for every~$\alpha<1$ (but not necessarily Lipschitz). This is due to an example constructed by Filonov~\cite{Fil} (see also~\cite{Mi,P}), which implies the existence of an elliptic equation with coefficients~$\a(\cdot)$ belonging to $\cap_{\alpha<1} C^{0,\alpha}$ which possesses a solution~$u$ in the whole space~$\Rd$ which is the product of~$\exp(\lambda x_{d})$, for some explicit $\lambda>0$, and a compactly supported function of the other $d-1$ variables.

\smallskip

The main ingredient in the proof of Theorem~\ref{t.threeball} is the large-scale analyticity estimate proved in our previous paper~\cite{AKS}.
This result is the ``best possible'' quantitative estimate in periodic homogenization, and states essentially that corrector expansions of arbitrary high order can approximate solutions of a periodic elliptic equation up to an error which is exponentially small in the aspect ratio. The statement of this result is recalled below in Proposition~\ref{p.analyticity}.
In that paper, we also proved a large-scale three-ball theorem as a corollary of large scale analyticity (see~\cite[Theorem 1.4]{AKS}) down to the optimal scale $C\log M$. However, in this estimate has sub-optimal exponent (like $\alpha< \sfrac12$ in~\eqref{e.free.ballz}) and for this reason the estimate is not particularly useful (and cannot for instance be iterated to yield a doubling inequality without  catastrophically blowing up the doubling ratio).
Recently,  Kenig, Zhu and Zhuge~\cite{KZZ} proved, by argument which was also based on~\cite{AKS}, a doubling estimate down to scale $r = M^\delta$ for arbitrary~$\delta>0$;  this is, however, still far from the optimal scale of $C\log M$, and their estimate for the doubling ratio is also much larger than the right side of~\eqref{e.doubling.logM}. As explained below, the scales of order~$\log M$ correspond to exponential growth, which is the critical regime for spectral theory problems. 

\smallskip

The goal of this work is therefore to improve the arguments passing from large-scale analyticity to quantitative unique continuation, demonstrating in particular that the former is powerful enough to imply quantitative unique continuation estimates at the optimal scale of~$C\log M$. The idea is to obtain a very sharp three-ellipsoid estimate down to scale~$C(\log M)^p$ for some finite~$p$ using precise estimates on~$\a(x)$-harmonic ``polynomials'' on these scales, and transferring such estimate to general solutions via large-scale analyticity. We then invoke large-scale analyticity a second time, in a different way, to jump between scales~$C(\log M)^p$ and~$C\log M$ (this is where the $C \log \log M$ in the exponent comes from). Here we use $p=4$, but the precise value of~$p$ is irrelevant since it disappears into the constant~$C$ next to the iterated logarithm. 

\smallskip

In particular, the constants in~\eqref{e.threeball.logM} and~\eqref{e.doubling.logM} and can be improved on scales larger than~$C (\log M)^4$. On such scales there is essentially no degradation of the doubling ratio and the prefactor in the three-ellipsoid estimate is close to~$1$. For scales closer to~$C\log M$, we are unsure whether the prefactor~$M^{C\log \log M}$ is sharp, or whether it can be reduced to $M^C$. However, this question reduces to one about the behavior of higher-order homogenized tensors and~$\mathscr{A}$-harmonic polynomials which seems difficult to resolve. See Remark~\ref{r.logM4} for more details. 

\smallskip

If the coefficient field is assumed to be Lipschitz, then one may use the classical frequency function estimate of Garofalo and Lin~\cite{GL} to pass from the minimal doubling scale of~$C\log M$ all the way down to infinitesimal scales. 
Of course, if $M$ is large, then starting from scale $C\log M$ with Lipschitz constant $L$ is the same, by scaling, as starting from scale one with Lipschitz constant~$CL \log M$.
Since the unique continuation estimates in~\cite{GL} degenerate rather badly as the Lipschitz constant becomes large, this leads to an estimate with a doubling ratio blowing up quite quickly in~$M$. 

\begin{corollary}[Every-scale doubling inequality]
\label{c.doubling.dumbdumb}
Suppose, in addition to the hypotheses of Theorem~\ref{t.threeball}, that the coefficients~$\a(\cdot)$ satisfy, for a given $L\in(0,\infty)$, the Lipschitz condition
\begin{align*}
\sup_{x,y\in \Rd,\, x\neq y}
\frac{ | \a(x) - \a(y)|}{|x-y|}
\leq 
L.
\end{align*}
Then there exists $C(d,\Lambda)<\infty$ such that we have the estimate
\begin{align}
\label{e.doubling.Lipschitz}
\sup_{ r \in (0,R]} 
\frac{\left\| u \right\|_{\underline{L}^2(E_{r})}}{\left\| u \right\|_{\underline{L}^2(E_{\theta r})}}
\leq 
\exp\left( CM^{CL \log\log M}\log M \right)
\,.
\end{align}
\end{corollary}

The estimate~\eqref{e.doubling.Lipschitz} is certainly not optimal and it actually contains much less information that the \emph{large-scale} doubling estimate~\eqref{e.doubling.logM}. 
We can see from Theorem~\ref{t.threeball} that there is very little degradation in the doubling ratio as we move down the scales, from the largest scale~$R$ to minimal scale~$r = C\log M$. 
Everything bad which can happen to the doubling ratio must occur on scales smaller than $C\log M$. Corollary~\ref{c.doubling.dumbdumb} fails to see this crucial fact, because it conflates \emph{large-scale estimates}, which are obtained from the periodicity of the coefficients, and \emph{small-scale estimates}, which are obtained from their regularity. 
It is important to keep these separate. 

\smallskip

Quantitative unique continuation results for periodic operators are of particular interest in the regime in which the function exhibits exponential growth, that is, on scales~$r$ for which~$r=O(\log M)$. 
One reason is because of possible applications to open problems in the spectral theory for such operators. For instance, it is still not known whether or not a general periodic, uniformly  elliptic  operator  may possess eigenvalues with~$L^2(\Rd)$ eigenfunctions (see for instance~\cite[Section 6]{Ku}). 
If $u\in H^1_{\mathrm{loc}}(\Rd)$ is a solution of
\begin{align*}
-\nabla \! \cdot \a\nabla u = \lambda u \quad \mbox{in} \ \Rd,
\end{align*}
then we may add the dummy variable $x_{d+1}$ by defining 
$$v(x,x_{d+1}) := \exp( \lambda^{\sfrac12} x_{d+1}) u(x)$$ 
to obtain a solution of a periodic equation in dimension~$d+1$ which grows like an exponential in the dummy variable direction. Obviously, the larger the eigenvalue~$\lambda$, the faster the exponential growth is. Therefore it would be of great interest to extend the estimate~\eqref{e.doubling.logM} to scales $r = \delta \log M$ for a tiny~$\delta>0$, because this would imply the nonexistence of eigenvalues with eigenfunctions in~$L^2(\Rd)$ high up in the spectrum (depending on~$\delta$). 
On the other hand, if $\lambda$ is sufficiently small, then we can apply the doubling inequality to get at least a \emph{one} smaller ellipsoid satisfying doubling, and this turns out to be enough to imply the following statement (which was already proved in~\cite[Theorem 1.5]{AKS} as a consequence of large-scale analyticity). 

\begin{corollary}
\label{c.embed}
There exist~$\lambda_0(d,\Lambda)>0$ such that, if 
\begin{equation*}
(\lambda,u) \in [0,\lambda_0]\times ( H^1_{\mathrm{loc}}(\Rd) \cap L^2(\Rd))
\end{equation*}
satisfies the eigenvalue equation
\begin{equation}
-\nabla \! \cdot \a\nabla u = \lambda u \quad \mbox{in} \ \Rd\,,
\end{equation}
then $u\equiv 0$ in $\Rd$. 
\end{corollary}

Perhaps the main open question regarding quantitative unique continuation for periodic elliptic operators left unresolved by this work is the following: assuming~$\a(\cdot)$ is at least Lipschitz, how does the doubling ratio behave on scales between~$o(\log M)$ and~$C\log M$?

\subsection*{Comparison to previous works}

The first quantitative unique continuation results in the context of periodic homogenization 
was proved by Lin and Shen~\cite{LS}, who obtained a doubling inequality like Corollary~\ref{c.doubling.dumbdumb} but with implicit dependence of the right side of~\eqref{e.doubling.Lipschitz} on the initial doubling ratio~$M$. 
Kenig and Zhu~\cite{KZ} subsequently proved a version of the three-ball inequality~\eqref{e.threeball.logM} on scales $r \geq M^C$ with sub-optimal exponents. In our previous paper~\cite{AKS}, we improved this result by using large-scale analyticity to obtain a a three-ball inequality is valid down to scale $C\log M$ with almost optimal exponents. 
Later, Kenig, Zhu and Zhuge~\cite[Theorem 1.1]{KZZ} proved, by argument also based on large-scale analyticity, a similar estimate to~\eqref{e.doubling.Lipschitz} with a right side which is doubly exponential in~$M$, namely $\exp( \exp(M^\delta))$ for small $\delta>0$. Their arguments essentially yield a version of the doubling inequality~\eqref{e.doubling.logM} down to scale $r = M^\delta$ for an arbitrarily small exponent~$\delta>0$. 

\smallskip

The paper~\cite{{KZZ}} also contains estimates on nodal sets for eigenfunctions with this double exponential dependence,~$\exp( \exp(M^\delta))$. Using our Corollary~\ref{c.doubling.dumbdumb}, this estimate can be improved to~$\exp( M^{C \log \log M})$.

\subsection*{Outline of the paper}

In Section~\ref{s.prelims} we set up the notation and recall some basic facts concerning the action of the Laplacian operator on the space of real polynomials. 
In Section~\ref{s.correctors}, we recall the definitions of the higher-order correctors and homogenized tensors in periodic homogenization and record some estimates on the difference between $\a(x)$-harmonic polynomials and~$\ahom$-harmonic polynomials. This leads to a three-ellipsoid estimate for $\a(x)$-harmonic polynomials in Proposition~\ref{p.threeball.Am}, which in  Section~\ref{s.mainproofs} is combined with large-scale analyticity to obtain Theorem~\ref{t.threeball}. 

\smallskip

We do not include the proofs of Corollaries~\ref{c.doubling.dumbdumb} and~\ref{c.embed} here. The proof of the former is obtained by well-known arguments: one combines~\eqref{e.doubling.logM} at scale~$C\log M$ with the aid of the Lipschitz assumption  and the result of Garofalo and Lin~\cite{GL} to handle the small scales. Since the explicit dependence of the doubling ratio on the Lipschitz constant is required, it is necessary to inspect the argument of~\cite{GL} and track this dependence line-by-line. This is routine and has essentially already been done\footnote{Apply~\cite[(8)]{LM} after rescaling.} in~\cite{LM}, so we omit the details.
Since Corollary~\ref{c.embed} was proved already in~\cite{AKS} by a simple argument which is similar to the one implied here, it has also been left to the reader.

\section{Preliminaries}
\label{s.prelims}

\subsection{Notation}
Throughout,~$d$ denotes the dimension of Euclidean space~$\Rd$ and~$\Lambda\in [1,\infty)$ is the ellipticity of the coefficient field~$\a(\cdot)$, which is assumed to satisfy~\eqref{e.ellipticity} and~\eqref{e.periodicity}. 
We let~$C$ and~$c$ denote positive constants which depend only on~$(d,\Lambda)$ and may vary in each occurrence. 

The set of natural numbers is~$\N$ and~$\N_0=\N \cup\{0\}$.
The standard basis for~$\Rd$ is~$\{e_1,\ldots,e_d\}$.
We  denote volume-normalized~$L^p$ norms with underlines:
if~$U\subseteq \Rd$ with~$|U| <\infty$, then 
\begin{align*}
\left\| u \right\|_{\underline{L}^p(U)} 
:= \biggl( 
\fint_U |u(x)|^p \,dx \biggr)^{\!\!\sfrac1p}
=
\biggl( 
\frac1{|U|} \int_U |u(x)|^p \,dx \biggr)^{\!\!\sfrac1p}
\,.
\end{align*}
We employ the usual multi-index notation and reserve the symbols~$\alpha,\beta\in \N_0^d$ to denote multi-indices and denote 
\begin{align*}
|\alpha| := \alpha_1+\ldots +\alpha_d
\qquad \mbox{and} \qquad 
\alpha!=\alpha_1!\cdots\alpha_d!
\,.
\end{align*}
Define, for~$m\in\N$ and~$|\alpha| = m$, the multinomial coefficient
\begin{align*}
\binom{m}{\alpha}
:=
\frac{m!}{\alpha!} \,.
\end{align*}
A \emph{symmetric tensor of order $m$} is a mapping
\begin{align*}
T: \{ \alpha \in \N_0^d \,:\, |\alpha|= m\} \to \R
.
\end{align*}
We write is in coordinate notation as $(T_\alpha)_{|\alpha|=m}$. We let $\T_m$ denote the set of symmetric tensors of order~$m$. We identify $\T_0$ with $\R$. We define an inner product on~$\T_m$ by
\begin{align*}
(S:T) 
:= 
\sum_{|\alpha|=m}
\binom{m}{\alpha} S_\alpha T_\alpha.
\end{align*}
We also denote $|T| := (T:T)^{\frac12}$. 
In this paper, we do not distinguish between tensors and their symmetric parts. Therefore the reader may assume that all tensors are symmetric, and that all tensor products are symmetrized. 

\smallskip

We let~$\P$ denote the linear space of real polynomials in~$\Rd$, the subspace of polynomials of degree at most $m\in\N$ is denoted by~$\P_m^*$, and $P_m^* \subseteq \P_m$ is the subspace of~$m$-homogeneous polynomials. For each $x\in\Rd$, we let $x^{\otimes m}$ denote the tensor with entries $x^\alpha$, $|\alpha|=m$. We can write any polynomial $p\in\mathbb{P}_m$ as 
\begin{align*}
p = \sum_{k=0}^m \frac1{k!} \nabla^k p(0):x^{\otimes k}. 
\end{align*}
The set of harmonic polynomials is denoted by
\begin{align*}
\H:= \left\{ p\in\P\,:\, \Delta p = 0 \right\}
\end{align*}
and we also set~$\H_m:= \H \cap \P_m$ and~$\H_m^* := \H \cap \P_m^*$.

\subsection{The Laplacian on the space of polynomials}
In this subsection we recall some basic facts regarding the action of the Laplacian~$\Delta$ on the space~$\P$ of polynomials. We define an inner product~$\langle\cdot,\cdot\rangle_{\P}$ on~$\P$ by 
\begin{align*}
\langle p, q \rangle_{\P}
=
\sum_{k=0}^\infty 
\frac1{k!} 
\nabla^kp(0) \!:\! \nabla^kq(0)
\,.
\end{align*}
Equivalently, for every pair of multi-indices $\alpha$ and $\beta$, 
\begin{align}
\label{e.monomials.orthogonal}
\bigl\langle x^\alpha, x^\beta \bigr\rangle_{\P}
=
\alpha! \indc_{\{\alpha=\beta\}}
\,.
\end{align}
The operator which performs multiplication by~$|x|^2$ and the Laplacian operator restriced to~$\P$ are adjoints with respect to this inner product:
\begin{align}
\label{e.Laplacian.mult.adjoint}
\bigl\langle |x|^2 p , q \bigr\rangle_{\P} = \bigl\langle p , \Delta q \bigr\rangle_{\P}
\,, \qquad \forall p,q\in\P
\,.
\end{align}
To see this, we compute, for any pair of multi-indices~$\alpha$ and~$\beta$, 
\begin{align*}
\bigl\langle |x|^2 x^\alpha, x^\beta \bigr\rangle_{\P}
&
=
\sum_{i=1}^d 
\bigl\langle x^{\alpha+2e_i}, x^\beta \bigr\rangle_{\P}
\\ & 
=
\sum_{i=1}^d
\beta!
\indc_{\{ \beta = \alpha+2e_i \}} 
\\ & 
= 
\sum_{i=1}^d
(\beta_i)(\beta_i-1) 
\alpha !
\indc_{\{ \beta = \alpha+2e_i \}} 
\\ & 
=
\sum_{i=1}^d
\bigl\langle x^{\alpha}, (\beta_i)(\beta_i-1) x^{\beta-2e_i} \bigr\rangle_{\P}
=
\bigl\langle  x^\alpha, \Delta( x^\beta)  \bigr\rangle_{\P}
\,.
\end{align*}
It follows from finite-dimensional linear algebra that $\P^*_m = \H_m^* \oplus (|x|^2 \P_{m-2}^*)$. In fact, by induction we obtain an orthogonal decomposition of $\P_m^*$:
\begin{align}
\label{e.Pstarm.decomp}
\P_m^* = 
\H_m^* \oplus |x|^2 \H_{m-2}^*
\oplus \cdots \oplus 
|x|^{2\lfloor \frac m2 \rfloor }\H_{m-2 \lfloor \frac m2 \rfloor}^*
\,.
\end{align}
By a direct computation, we have that
\begin{align}
\label{e.laplacian.mult.commutator}
\Delta (|x|^{k+2}p) 
=
(k+2)(d+2m+k) |x|^k p 
\,,
\quad \forall p \in\H_m^*\,.
\end{align}
In view of~\eqref{e.Pstarm.decomp}, 
we define an operator $S: \P_m^* \to \P_{m+2}^*$ by 
\begin{align}
\label{e.S.def}
S\Biggl( \sum_{k=0}^{\lfloor \frac m2\rfloor} 
|x|^{2k} p_k \Biggr)
:= 
\sum_{k=0}^{\lfloor \frac m2 \rfloor} 
b_k^{-1} 
|x|^{2k+2} p_k, 
\quad p_k \in \H_{m-2k}^*\,,
\end{align}
where the coefficients $b_k$ are given by
\begin{align}
\label{e.bk}
b_k := (2k+2)(d+2m-2k)
\,.
\end{align}
We extend~$S$ to~$\P$ by linearity. 
By a direct computation, using~\eqref{e.laplacian.mult.commutator}, we have  
\begin{align*}
\Delta Sp = p, \quad \forall p\in\P. 
\end{align*}
It is immediate from the definition of~$S$ and~\eqref{e.Laplacian.mult.adjoint} that~$\range(S) \perp \H$.
In other words,~$q=Sp$ is the unique polynomial solution of~$\Delta q = p$ such that~$q$ is orthogonal to all harmonic polynomials with respect to~$\langle \cdot, \cdot \rangle_{\P}$. 
To estimate the norm of~$S$, we fix~$p\in\P_{m}^*$ and write it via~\eqref{e.Pstarm.decomp} as
\begin{align*}
p 
=
\sum_{k=0}^{\lfloor \frac m2\rfloor} 
|x|^{2k} p_k,
\qquad p_k\in\H_{m-2k}^*, 
\end{align*}
and compute, using~\eqref{e.Laplacian.mult.adjoint} and~\eqref{e.Pstarm.decomp}, 
\begin{align*}
\langle Sp,Sp \rangle 
&
=
\sum_{k=0}^{\lfloor \frac m2\rfloor} 
b_{k}^{-2} 
\bigl \langle 
|x|^{2k+2} p_k ,
|x|^{2k+2} p_k
\bigl \rangle_\P 
\\ & 
=
\sum_{k=0}^{\lfloor \frac m2\rfloor} 
b_{k}^{-2} 
\bigl \langle 
\Delta \big( |x|^{2k+2} p_k\big) ,
|x|^{2k} p_k
\bigl \rangle_\P 
\\ & 
=
\sum_{k=0}^{\lfloor \frac m2\rfloor} 
b_{k}^{-2} 
\bigl \langle 
b_k |x|^{2k} p_k ,
|x|^{2k} p_k
\bigl \rangle_\P 
=
\sum_{k=0}^{\lfloor \frac m2\rfloor} 
b_{k}^{-1} 
\bigl \langle 
|x|^{2k} p_k ,
|x|^{2k} p_k
\bigl \rangle_\P 
\,.
\end{align*}
In particular, 
\begin{align}
\label{e.inversionbound}
\left\| Sp \right\|_{\P} 
\leq 
(4m+2d)^{-\frac12} 
\left\| p \right\|_{\P}
\,, \quad \forall p\in\P_m^*\,.
\end{align}
We will make use of the~$L^2$ Markov inequality for real multivariate polynomials, which states that 
\begin{align}
\label{e.reversePoincare}
\left\| \nabla p \right\|_{L^2(B_r)}
\leq 
\frac{m^2}{r}
\left\| p \right\|_{L^2(B_r)}
\,, \qquad 
\forall p\in\P_m
\,.
\end{align}
A proof of~\eqref{e.reversePoincare} can be found in~\cite{Ditz}.

\begin{lemma}
\label{l.inversion.Psi}
For every~$m\in\N_0$,~$p\in\P_m^*$ and~$r>0$
\begin{align}
\label{e.Laplacian.inversion}
\left\| S p \right\|_{L^2( B_r )}
\leq 
\frac{r^2}{m+1} 
\left\| p \right\|_{L^2( B_r)} 
\,.
\end{align}
\end{lemma}
\begin{proof}
Let $p\in\P_m^*$ and, using~\eqref{e.Pstarm.decomp}, write
\begin{align*}
p = 
\sum_{k=0}^{\lfloor \frac m2\rfloor} 
|x|^{2k} p_k,
\quad p_k \in \H_{m-2k}^*\,,
\end{align*}
and let~$q :=  Sp$. 
We will use the fact that the orthogonal decomposition~\eqref{e.Pstarm.decomp} is valid with respect to the inner product~$L^2(B_r)$, for any $r>0$. We compute 
\begin{align*}
\left\| q \right\|_{L^2(B_r)}^2
&
=
\sum_{k=0}^{\lfloor \frac m2 \rfloor} 
b_k^{-2} 
\left\| |x|^{2k+2} p_k \right\|_{L^2(B_r)}^2
\\ & 
\leq
\sum_{k=0}^{\lfloor \frac m2 \rfloor} 
b_k^{-2} 
r^4
\left\| |x|^{2k} p_k \right\|_{L^2(B_r)}^2
\\ & 
= 
\sum_{k=0}^{\lfloor \frac m2 \rfloor} 
\biggl( 
\frac{ r^4 }{(2k+2)(d+2m-2k) }
\biggr)^{\!\!2}
\left\| |x|^{2k} p_k \right\|_{L^2(B_r)}^2
\\ & 
\leq
\frac{r^4}{(4m+2d)^2}
\sum_{k=0}^{\lfloor \frac m2 \rfloor} 
\left\| |x|^{2k} p_k \right\|_{L^2(B_r)}^2
=
\frac{r^4}{(4m+2d)^2}
\left\| p\right\|_{L^2(B_r)}^2
\,.
\end{align*}
The proof is complete.
\end{proof}

\section{\texorpdfstring{$\mathscr{A}$}{script-A}-harmonic and \texorpdfstring{$\a(x)$}{a(x)}-harmonic polynomials}
\label{s.correctors}

\subsection{Higher-order correctors and homogenized tensors}

We introduce the higher-order correctors~$\phi_m$ and homogenized tensors~$\ahom_m$ arising in the theory of periodic homogenization. The corrector~$\phi_m$ of order $m\in\N$ is a~$\T_m$-valued,~$\Zd$-periodic function and the homogenized tensor~$\ahom_m$ belongs to~$\T_m$. These are completely characterized by the following:

\begin{itemize}
\setlength{\itemsep}{0.2em}
\item The zeroth order homogenized tensor vanishes: $\ahom_0= 0$.

\item The zeroth order corrector is equal to one: $\phi_0=1$. For convenience we also define $ \phi_{-1}:=0$.

\item For every $m\geq 1$, the homogenized tensor of order $m$ is defined in terms of the correctors of orders $m-1$ and $m-2$ by
\begin{align}
\label{e.ahomm}
\ahom_m:= \big\langle 
\a\nabla \phi_{m-1}
+
\a \phi_{m-2} 
\big\rangle.
\end{align}
Here $\langle\cdot \rangle$ denotes the mean of a periodic function.

\item For each $m\geq1$, the function $\phi_m$ is defined in terms of the 
correctors and homogenized tensors of orders~$m-1$ and~$m-2$
as the unique $\Zd$-periodic solution of the equation
\begin{align}
\label{e.phim}
\left\{
\begin{aligned}
&
-\nabla \! \cdot \a\nabla  \phi_{m} 
=
\nabla \cdot (\a   \phi_{m-1})
+
\a\nabla  \phi_{m-1}
+
\a   \phi_{m-2} 
-
\ahom_m
\,, 
\\ & 
\langle \phi_m \rangle = 0
\,.
\end{aligned}
\right.
\end{align}
\end{itemize}

To keep our computations short, we use the convention that, in expressions such as~\eqref{e.ahomm} and~\eqref{e.phim}, the symbol~$\nabla$ always contracts against the matrix~$\a$, while other implied products involving tensors denote the symmetrized tensor product. In particular, all tensors should be assumed to be symmetric even if this is not obvious from some of the expressions as written.  

Note that the equation for $\phi_m$ does posses a solution, which is unique up to additive constants, since the right-hand side is of zero mean by the definition of~$\ahom_m$ in~\eqref{e.ahomm}. 

The above recursively defines the higher-order correctors and homogenized tensors. 
It is furthermore easy to obtain, by induction, the existence of a constant~$C(d)<\infty$ such that, for every $m\in\N$, 
\begin{align}
\label{e.warmup.bounds}
\| \nabla \phi_m \|_{L^2(Q_1)} 
+
| \ahom_m | 
\leq 
(C\Lambda)^m. 
\end{align}
The De Giorgi-Nash $L^\infty$ estimate then yields
\begin{align*}
\| \phi_m \|_{L^\infty(\Rd)} 
\leq 
(C\Lambda)^{m+\frac d2}. 
\end{align*}
Allowing the constant to depend on~$\Lambda$, we obtain~$C(d,\Lambda)<\infty$ such that
\begin{align}
\label{e.corrector.bounds}
\| \phi_m \|_{L^\infty(\Rd)} 
+
| \ahom_m | 
\leq 
C^m\,,
\quad \forall m\in\N\,.
\end{align}
For more on the higher-order correctors and homogenized tensors, including an alternative derivation, see~\cite{AKS}.

\smallskip

The second-order homogenized tensor~$\ahom_2$ coincides with the usual homogenized matrix in the theory of periodic homogenization. Therefore we drop the subscript and write simple~$\ahom:= \ahom_2$. Recall that $\ahom$ satisfies the bounds
\begin{align*}
I_d \leq \ahom \leq \Lambda I_d. 
\end{align*}

We will argue next that the homogenized tensors of odd order vanish:
for every $n\in\N$, 
\begin{align}
\label{e.oddvanish}
\ahom_{2n+1} = 0.
\end{align}
This is a consequence of the following identity: for every $n\in\N$, 
\begin{align}
\label{e.magic.identity}
\left\langle 
\nabla \phi_n \a\nabla \phi_m -  \phi_{n-1} \a \phi_{m-1} 
\right\rangle
=
- \left\langle 
\nabla  \phi_{n+1} \a\nabla  \phi_{m-1} -  \phi_{n} \a \phi_{m-2} 
\right\rangle
\,.
\end{align}
To prove~\eqref{e.magic.identity}, we test the equation for $\phi_m$ with $\phi_n$ to obtain
\begin{align*}
\left\langle 
\nabla \phi_n \a\nabla \phi_m  
\right\rangle
&
=
\left\langle 
-\nabla \phi_n \a \phi_{m-1} + \phi_n \a\nabla  \phi_{m-1} 
+\phi_n \a \phi_{m-2} 
-\ahom_m\phi_n
\right\rangle
\\ & 
=
\left\langle 
-\nabla \phi_n \a \phi_{m-1} + \phi_n \a\nabla  \phi_{m-1} 
+\phi_n \a \phi_{m-2} 
\right\rangle
\,.
\end{align*}
We then test the equation for $ \phi_{n+1}$ with $ \phi_{m-1}$ to get
\begin{align*}
\left\langle 
\nabla  \phi_{m-1} \a\nabla  \phi_{n+1}  
\right\rangle
&
=
\left\langle 
-\nabla  \phi_{m-1} \a \phi_{n} +  \phi_{m-1} \a\nabla  \phi_{n} 
+ \phi_{m-1} \a \phi_{n-1} 
-\ahom_{n+1}  \phi_{m-1}
\right\rangle
\\ & 
=
\left\langle 
-\nabla  \phi_{m-1} \a \phi_{n} +  \phi_{m-1} \a\nabla  \phi_{n} 
+ \phi_{m-1} \a \phi_{n-1} 
\right\rangle
\,.
\end{align*}
Summing the previous two displays yields~\eqref{e.magic.identity}. 

\smallskip

Using that~$\phi_0=1$ and testing the equation for $\phi_1$ with $ \phi_{n-1}$, we get 
\begin{align*}
\ahom_n 
& 
=
\left\langle
\a\nabla \phi_{n-1} + \a \phi_{n-2} 
\right\rangle
=
-
\left\langle
\nabla \phi_1 \a\nabla \phi_{n-1} 
-
\phi_0 \a \phi_{n-2} 
\right\rangle
\,.
\end{align*}
Iterating~\eqref{e.magic.identity} then yields, for every~$k\in\N$ with~$1\leq k\leq n-1$, 
\begin{align}
\label{e.charlie.tock}
\ahom_n 
& 
=
(-1)^{k}
\left\langle
\nabla  \phi_{k} \a\nabla \phi_{n-k} -  \phi_{k-1} \a \phi_{n-k-1} 
\right\rangle
\,.
\end{align}
This yields a formula for the homogenized tensors of even order: for every $n\in\N$, 
\begin{align*}
\ahom_{2n} 
& 
=
(-1)^{n}
\left\langle
\nabla  \phi_{n} \a\nabla \phi_{n} 
-
 \phi_{n-1} \a \phi_{n-1} 
\right\rangle
\,.
\end{align*}
For homogenized tensors of odd order, 
we observe that 
the special case of $m=n+1$ in~\eqref{e.magic.identity} is
\begin{align}
\label{e.charlie.oddsout}
\left\langle 
\nabla \phi_n \a\nabla  \phi_{n+1} -  \phi_{n-1} \a \phi_{n} 
\right\rangle 
=
0.
\end{align}
We then apply~\eqref{e.charlie.tock} to obtain~\eqref{e.oddvanish}.

\subsection{The homogenized operator~\texorpdfstring{$\mathscr{A}$}{script-A}}

The main point of defining the correctors and homogenized tensors in the way we did in~\eqref{e.ahomm} and~\eqref{e.phim} above is to ensure the validity of the following statement. 
For any~$m\in\N$ and~$p\in \P_m$,
\begin{align}
\label{e.polyexact}
\psi := 
\sum_{m=0}^\infty
\nabla^m p \!:\! \phi_m
\quad\mbox{satisfies} \quad
-\nabla \! \cdot \a\nabla \psi = \mathscr{A} p \quad
\mbox{in} \ \Rd\,,
\end{align}
where $\mathscr{A}p$ is a polynomial belonging to~$\P_{m-2}$ and is defined by
\begin{align}
\label{e.mathscrA.def}
(\mathscr{A}p)(x)
:=
- \sum_{k=1}^\infty \ahom_{2k} \!:\! \nabla^{2k} p(x)
\,.
\end{align}
This assertion can be checked by a direct computation, using only~\eqref{e.ahomm}--\eqref{e.phim}. 

\smallskip

We view $\mathscr{A}$ as an operator on space~$\P$ of real polynomials, and we call it the \emph{homogenized operator}. It can be written as the sum of the ``usual'' homogenized operator $-\nabla \! \cdot \ahom\nabla$ and higher-order terms:
\begin{align*}
\mathscr{A} = -\nabla\! \cdot \ahom \nabla + \mathscr{A}^\prime, 
\quad \mbox{where} \quad
\mathscr{A}^\prime 
:= 
-\sum_{k=2}^\infty  \ahom_{2k} \!:\! \nabla^{2k} \,.
\end{align*}
It is customary in the theory of periodic homogenization to ignore the higher-order part~$\mathscr{A}^\prime$ of the homogenized operator, and this suffices to many purposes. 
However, if precise large-scale regularity estimates or quantitative unique continuation estimates are desired, it is not always wise to substitute the ``real'' homogenized operator~$\mathscr{A}$ by its leading-order approximation,~$-\nabla\! \cdot \ahom\nabla$. 

\subsection{Harmonic approximation of \texorpdfstring{$\mathscr{A}$}{script-A}-harmonic polynomials}

If we want to find solutions of the equation $-\nabla \cdot \a \nabla \psi = 0$ with polynomial growth at infinity, then in view of~\eqref{e.polyexact} we should first find polynomials~$p$ satisfying the homogenized equation
\begin{align*}
\mathscr{A} p = 0. 
\end{align*}
It turns out that the linear space 
\begin{align*}
\left\{ p\in\P_m \,:\, \mathscr{A} p = 0 \right\}
\end{align*}
is a linear subspace of $\P_m$ with the same dimension as~$\H_m$. A proof can be found in~\cite{AKS}, for instance, or extracted from the proof of the following lemma.

\begin{lemma}[Harmonic approximation in $L^2(E_r)$]
\label{l.almost.identitymap.ballz}
There exists~$C(d,\Lambda)<\infty$ such that, for~$m\in\N$, $r\geq Cm^{\sfrac 72}$ and~$q\in \P_m$ satisfying~$\mathscr{A} q = 0$, there exists~$p\in \P_m$ such that $\nabla \cdot\ahom \nabla p=0$ and
\begin{align}
\label{e.project.scrA.H.ballaz}
\| q - p \|_{L^2(E_r)}
\leq 
\frac{Cm^{\sfrac {15}2}}{r^2}
\| q \|_{L^2(E_r)}
\,.
\end{align}

\end{lemma}
\begin{proof}
We suppose that~$\ahom=I_d$ so that $E_r=B_r$.

\smallskip

Let $p\in\H_m^*$. 
We recursively define sequences~$\{ p_k\}\subseteq \P$ and $\{q_k\} \subseteq \P$ by
\begin{align}
\label{e.recursion.S}
\left\{
\begin{aligned}
&
p_0 = 0, \quad q_0 = p, \\
& 
p_k =  \sum_{j=2}^{k+1} 
\ahom_{2j} 
\!:\! \nabla^{2j} q_{k+1-j} , \\
& 
q_k = - S (p_k)
\,.
\end{aligned}
\right.
\end{align}
In particular, we have that
\begin{align}
\label{e.Delta.pk.E}
-\Delta q_{k} 
= 
p_k
\,.
\end{align}
Note that~$q_0=p\in \H^*_m \subseteq \P_m^*$.
Since~$S(\P_n^*)\subseteq \P_{n+2}^*$ for all $n\in\N_0$, by induction we have that~ $q_k \in \P_{m-2k}^*$ and $p_k\in \P_{m-2k-2}^*$ for every $k\in\N_0$. Moreover, we have that~$p_k = q_k = 0$ for every $k \geq \lfloor \frac m2 \rfloor $.
Summing~\eqref{e.Delta.pk.E} over~$k$, we find that the function~$q\in \P_m$ defined by
\begin{align*}
q := 
\sum_{k=0}^{\lfloor \frac m2 \rfloor-1}
q_k
\end{align*}
satisfies
\begin{align}
\label{e.Deltaq.compute}
-\Delta q 
= 
\sum_{k=1}^{\lfloor \frac m2 \rfloor-1}
p_k
&
=
\sum_{k=1}^{\lfloor \frac m2 \rfloor-1}
\sum_{j=2}^{k+1}
\ahom_{2j} \!:\! \nabla^{2j} q_{k+1-j}
\\ & \notag
=
\sum_{j=2}^{\lfloor \frac m2 \rfloor}
\sum_{k=j-1}^{\lfloor \frac m2 \rfloor-1}
\ahom_{2j} \!:\! \nabla^{2j} q_{k+1-j}
\\ & \notag
=
\sum_{j=2}^{\lfloor \frac m2 \rfloor}
\ahom_{2j} \!:\! \nabla^{2j}
\Biggl( \sum_{k=0}^{\lfloor \frac m2 \rfloor-j}
 q_{k} \Biggr)
=
\mathscr{A}'q
.
\end{align}
That is,~$\mathscr{A} q = 0$.

\smallskip

We next estimate the~$L^2(B_r)$ norm of each~$q_k$ and hence~$q-p$. 
Using~\eqref{e.reversePoincare} and the triangle inequality, we  have  
\begin{align*}
\left\| p_k \right\|_{L^2(B_r)}
& 
\leq
\biggl\| 
\sum_{j=2}^{k+1}
\ahom_{2j} \!:\! \nabla^{2j} q_{k+1-j}
\biggr\|_{L^2(B_r)} 
\\ & 
\leq
\sum_{j=2}^{k+1} 
\bigl| \ahom_{2j} \bigr|
\bigl\| 
\nabla^{2j} q_{k+1-j}
\bigr\|_{L^2(B_r)} 
\\ &
\leq
\sum_{j=2}^{k+1} 
\left(\frac{C(m-2(k+1-j))^2}{r} \right)^{2j}
\bigl\| 
q_{k+1-j}
\bigr\|_{L^2(B_r)} 
\,.
\end{align*}
Using this and~\eqref{e.Laplacian.inversion}, we obtain, for every $k\leq \lfloor \frac m2 \rfloor -1$,  
\begin{align}
\| q_{k} \|_{L^2(B_r)} 
&
\leq 
\frac{r^2}{(m-2k-1)} 
\| p_{k} \|_{L^2(B_r)} 
\\ & \notag
\leq
\frac{r^2}{(m-2k-1)} 
\sum_{j=2}^{k+1} 
\biggl(\frac{C(m-2k+2j)^2}{r} \biggr)^{\!2j}
\bigl\| 
q_{k+1-j}
\bigr\|_{L^2(B_r)} 
\,.
\end{align}
If~$r\geq Cm^{\sfrac {7}2}$ for a sufficiently large constant~$C$, it then follows by induction that, for every $k\in\{0,\ldots ,\lfloor \frac m2\rfloor-1\}$, 
\begin{align*}
\| q_{k} \|_{L^2(B_r)} 
\leq 
\biggl(\frac{C(m-2k)^{7}}{r^2} \biggr)^{\!k}
\| q_{0} \|_{L^2(B_r)} 
=
\biggl(\frac{C(m-2k)^{7}}{r^2} \biggr)^{\!k}
\| p \|_{L^2(B_r)}
\,.
\end{align*}
Squaring and summing this over~$k\in \{1,\ldots,\lfloor \frac m2\rfloor-1\}$ yields
\begin{align}
\label{e.projection.scrA.H.0}
\| q - p \|_{L^2(B_r)} 
\leq
\frac{Cm^{7}}{r^2}
\| p \|_{L^2(B_r)}
\,.
\end{align}

\smallskip

The mapping $p \mapsto q$ defined above is clearly linear on~$\H_m^*$ and may be extended to a linear operator on~$\H$. Let us denote this linear operator by~$K_r$. 
Fix a general (possibly nonhomogeneous)~$p\in\H_m$ and set $q= K_r(p)$. Write
\begin{align*}
p = \sum_{n=0}^m p^{(n)}, \quad p^{(n)} \in \H_n^*
\end{align*}
and set $q^{(n)} = K_r(p^{(n)})$. 
By the triangle inequality, we have
\begin{align*}
\| q - p \|_{L^2(E_r)} 
& 
\leq
\sum_{n=0}^{m}
\bigl\| q^{(n)} - p^{(n)} \bigr\|_{L^2(B_r)}
\\ & 
\leq 
\sum_{n=0}^{m}
\frac{Cn^{7}}{r^2}
\bigl\| p^{(n)} \bigr\|_{L^2(B_r)}
\\ & 
\leq 
\biggl( 
\sum_{n=0}^{m}
\frac{Cn^{14}}{r^4}
\biggr)^{\!\!\sfrac12} 
\biggl( 
\sum_{n=0}^{m}
\bigl\| p^{(n)} \bigr\|_{L^2(B_r)}^2
\biggr)^{\!\!\sfrac12} 
=
\biggl( 
\frac{Cm^{15}}{r^4}
\biggr)^{\!\!\sfrac12} 
\| p \|_{L^2(B_r)} 
\,.
\end{align*}
This completes the proof of the lemma. 
\end{proof}

\subsection{Heterogeneous polynomials}

The correctors give us a parametrization of a large-family of solutions of the heterogeneous equation~\eqref{e.pde} in~$\Rd$ which turns out to be precisely the set of the solutions which grow at most like polynomial at infinity. 

\smallskip

By~\eqref{e.polyexact}, for every polynomial~$q\in\Ahom_m$, 
the function defined by 
\begin{align}
\label{e.psi.rep}
\psi (x) := 
\sum_{m=0}^\infty
\nabla^m q(x) \!:\! \phi_m(x)
\end{align}
is a solution of
\begin{align}
\label{e.pde.again}
- \nabla \! \cdot \a\nabla \psi=0
\quad \mbox{in} \ \Rd
\,.
\end{align}
In the previous section, we found an explicit and complete parametrization of the~$\mathscr{A}$-harmonic polynomials, putting them into one-to-one correspondence with the~$\ahom$-harmonic polynomials with quantitative bounds on the differences between these. 

\smallskip

We denote the set of~$\mathscr{A}$-harmonic polynomials of degree at most~$m\in\N$ by 
\begin{align}
\Ahom_{m} 
:= 
\bigl\{ q \in \P_m 
\,:\, 
\mathscr{A}q = 0 
\bigr\}
\,.
\end{align}
Likewise, we denote the space of $\a(x)$-harmonic ``heterogeneous polynomials'' of degree at most~$m\in\N$ by 
\begin{align}
\A_m := 
\biggl\{ 
\psi \in H^1_{\mathrm{loc}}(\Rd)  \, :\, 
\psi = \sum_{m=0}^\infty
\nabla^m q \!:\! \phi_m \,,
\ q \in \Ahom_m 
\biggr\}
\,.
\end{align}
We also put $\Ahom = \cup_{m\in\N} \Ahom_m$ and $\A := \cup_{m\in\N} \A_m$. 

\smallskip

We next present an estimate on the difference between an element~$\psi \in \A$ and the~$\mathscr{A}$-harmonic polynomial in its representation~\eqref{e.psi.rep}.

\begin{lemma}
\label{l.psi.to.mathscrA}
There exists~$C(d,\Lambda)<\infty$ such that, for every~$m\in\N$, $r\geq Cm^2$ and~$\psi\in \A_m$ and~$q\in \Ahom_m$ satisfying 
\begin{align}
\label{e.psi.q}
\psi = \sum_{n=0}^m \nabla^n q  \!:\! \phi_n
\,,
\end{align}
we have the estimate
\begin{align}
\label{e.psi.to.mathscrA}
\left\| \psi - q \right\|_{L^2(E_r)}
\leq 
\frac {Cm^{2}}{r}
\left\| \psi \right\|_{L^2(E_r)}
\,.
\end{align}
Consequently, for every~$\psi\in \A_m$, there exists~$p\in \P_m$ such that~$\nabla \cdot \ahom\nabla p=0$ and, for every~$m\in\N$ and~$r\geq Cm^{4}$,
\begin{align}
\label{e.project.scrA.H}
\left\| \psi - p \right\|_{L^2(E_{r})}
\leq 
\frac {Cm^4 }{r}
\left\| \psi \right\|_{L^2(E_{r})}
\,.
\end{align}
\end{lemma}
\begin{proof}
By~\eqref{e.reversePoincare} and~\eqref{e.warmup.bounds}, we have 
\begin{align*}
\left\| \psi - q \right\|_{L^2(E_{r})}
&
\leq
\sum_{n=1}^m 
\bigl\| \phi_n \bigr\|_{L^\infty(\Rd)}
\left\| \nabla^n q \right\|_{L^2(E_{R})}
\leq 
\sum_{n=1}^m 
\biggl( \frac{Cm^2}{r} \biggr)^{\!\! n} 
\left\| q \right\|_{L^2(E_{R})}
\,.
\end{align*}
If~$r \geq Cm^{2}$ for $C$ sufficiently large, then 
\begin{align*}
\sum_{n=1}^m 
\biggl( \frac{Cm^2}{r} \biggr)^{\!\! n}
\leq
\frac{Cm^{2}}{r}
\end{align*}
and so we obtain~\eqref{e.psi.to.mathscrA} for every $r\geq Cm^{2}$.

\smallskip

The estimate~\eqref{e.project.scrA.H} is a consequence of Lemma~\ref{l.almost.identitymap.ballz},~\eqref{e.psi.to.mathscrA} and the triangle inequality.
\end{proof}

As a consequence of the previous lemma, we can prove a three-ball inequality for elements of~$\A_m$. 

\begin{proposition}
\label{p.threeball.Am}
There exists $C(d,\Lambda)<\infty$ such that, for every $m\in\N$, $\psi\in\A_m$~$\theta \in (0,1)$ and $r\geq C \theta^{-2} m^4$,
we have the estimate
\begin{align}
\label{e.threeball.heat.Am}
\frac
{\left\| \psi \right\|_{L^2(E_{\theta r})}^2}
{\left\| \psi \right\|_{L^2(  E_{\theta^2 r})}\left\| \psi \right\|_{L^2(E_{r})}}
\leq 
1 + \frac{Cm^4}{\theta^2 r}
\,.
\end{align}
\end{proposition}
\begin{proof}
According to Lemma~\ref{l.psi.to.mathscrA},
we may select~$p\in\P_m$ with $\nabla \cdot\ahom\nabla p=0$ such that,
for every $r\geq Cm^4$, 
\begin{align*}
\left\| \psi - p \right\|_{L^2(E_{r})}
\leq 
\frac {Cm^4 }{r}
\left\| \psi  \right\|_{L^2(E_{r})}
\,.
\end{align*}
Using the triangle inequality and the three-ball inequality for $\ahom$-harmonic functions, 
we obtain, 
for every~$r \geq C \theta^{-2} m^4$,
\begin{align*}
\left\| \psi \right\|_{L^2(E_{\theta r})}^2
&
\leq
\biggl( 1 + \frac{Cm^4}{\theta r} \biggr)^{\!\!2}
\left\| p \right\|_{L^2(E_{\theta r})}^2 
\\ &
\leq 
\biggl( 1 + \frac{Cm^4}{\theta r} \biggr)^{\!\!2}
\left\| p \right\|_{L^2(E_{\theta^2 r})} 
\left\| p \right\|_{L^2(E_{r})}
\\ & 
\leq
\biggl( 1 + \frac{Cm^4}{\theta^2 r} \biggr)^{\!\!4} \!
\left\| \psi \right\|_{L^2(E_{\theta^2 r})} 
\left\| \psi \right\|_{L^2(E_{r})}
\,.
\end{align*}
This completes the proof. 
\end{proof}

A classical ``Liouville-type'' result of Avellaneda and Lin~\cite{AL2} asserts that every solution~$\psi$ of~\eqref{e.pde.again} which grows at most like a polynomial at infinity must be an element of~$\A$. 
A quantitative version of this Liouville result, known as the \emph{large-scale~$C^{m,1}$ estimate} in homogenization, gives an estimate of how well an arbitrary solution of~\eqref{e.pde} in a large (but finite) ball can be approximated by an element of~$\A_m$ in a smaller ball. The main result of~\cite{AKS} was an optimal quantification of this estimate---in the dependence of the prefactor constant and the minimal scale on the degree~$m$ of the polynomial---which we call \emph{large-scale analyticity}. We next recall the statement of this result stated here in terms of the ellipsoids $\{ E_r\}$ rather than cubes as in~\cite{AKS}, which is equivalent.

\begin{proposition}[{\cite[Theorem 1.1]{AKS}}]
\label{p.analyticity}
There exists a constant~$C(d,\Lambda)<\infty$ such that, for every $m\in\N$, $R\in [Cm,\infty)$ and solution $u\in H^1(E_R)$ of
\begin{equation}
-\nabla \cdot \a\nabla u = 0 \quad \mbox{in} \ E_R,
\end{equation}
there exists $\psi \in \A_{m}$ such that, for every $r\in \left[Cm, R\right]$,
\begin{equation}
\label{e.Cm1}
\left\| u - \psi \right\|_{\underline{L}^2(E_r)} 
\leq 
\left( \frac{Cr}{R} \right)^{m+1} \left\| u \right\|_{\underline{L}^2(E_R)}. 
\end{equation}
\end{proposition}

The polynomial~$\psi\in\A_m$ in the statement of Proposition~\ref{p.analyticity} satisfies the following growth estimate: for every~$r\in [Cm,R]$ and~$s\in [Cm,\infty)$,
\begin{align}
\label{e.growth.psi}
\left\| \psi \right\|_{\underline{L}^2(E_s)}
\leq
\biggr( \frac{Cs}{r} \biggr)^{\!\!m+1}
\left\| \psi \right\|_{\underline{L}^2(E_r)} 
\leq
\biggr( \frac{Cs}{r} \biggr)^{\!\!m+1}
\left\| u \right\|_{\underline{L}^2(E_r)} 
\,.
\end{align}
The first inequality is a consequence of~\cite[Lemmas 2.7 \& 3.1]{AKS} and the second is immediate from~\eqref{e.Cm1} with~$R=r$.

\section{Doubling and three-ball inequalities at large scales}
\label{s.mainproofs}

We now present the proof of Theorem~\ref{t.threeball}, which breaks into two steps. 
In the first step, we take advantage of the better harmonic approximation (thanks to Lemma~\ref{l.psi.to.mathscrA}) on scales larger than~$C(\log M)^4$ to obtain more precise estimates. By performing an iteration down the scales, on these ``very large'' scales we obtain a three-ellipsoid estimate with constant close to one and a doubling estimates with ratio~$M^C$.
In the second step of the proof, we pass crudely from scale $C(\log M)^4$ to scale~$C\log M$ using an argument more similar to~\cite[Theorem 1.4]{AKS}---but as this involves only~$C \log \log M$ many scales, we escape with only a slight inflation of the doubling ratio, from $M^C$ to $M^{C\log\log M}$. 

\smallskip

\begin{proof}[{Proof of Theorem~\ref{t.threeball}}]

Let~$\theta \in (0,\sfrac12]$ and $K,L\in [2,\infty]$ be a fixed constants which will be selected below and will depend only on~$(d,\Lambda)$.  
Let $M \in [5,\infty)$, $R \in [ C K (\log M)^4, \infty)$ and $u\in H^1(E_R)$ be a solution of 
\begin{align*}
-\nabla \cdot \a \nabla u = 0 \quad \mbox{in} \ E_R
\end{align*}
which satisfies the doubling assumption
\begin{align}
\label{e.double.ass}
\left\| u \right\|_{\underline{L}^2(E_{R})}
\leq
M \left\| u \right\|_{\underline{L}^2(E_{\theta^6 R})} 
\,.
\end{align}
Note that, unlike~\eqref{e.initial.doubling}, we have~$\theta^6R$ instead of~$\theta R$ on the right side. It suffices to prove the theorem under the stronger assumption of~\eqref{e.double.ass}, since we can then replace~$\theta$ by~$\theta^6$ and replace~$C$ by $6C$.

\smallskip

We argue by induction, going down the scales. Define, for each $r\in[1,R]$,
\begin{align*}
N(r):= \sup_{s \in [r , R]}
\frac{\left\| u \right\|_{\underline{L}^2(E_{s})}}{ \left\| u \right\|_{\underline{L}^2(E_{\theta s})} }
\,.
\end{align*}
Suppose that~$r_0 \in [1,R]$ and $L \in [1,\infty)$ satisfy
\begin{align}
\label{e.double.induct.hyp}
r_0 \in [ C \theta^{-2}  K^4 (\log M)^4, \theta^5 R]
\quad \mbox{and} \quad
N(r_0)
\leq 
LM.
\end{align}
By the doubling assumption~\eqref{e.double.ass}, we have that~\eqref{e.double.induct.hyp} is valid for~$r_0 = \theta^5 R$ and $L=\theta^{-3d}$. 
We will argue that~\eqref{e.double.induct.hyp} implies that 
\begin{align}
\label{e.double.induct.wrap}
N(\theta r_0) \leq \biggl(1 + \frac{CK^4 (\log M)^4}{r_0} + M^{-100 \log\frac{R}{r_0}} \biggl)  N(r_0) 
\,,
\end{align}
provided that~$\theta$ is taken sufficiently small, depending only on~$(d,\Lambda)$, and~$K$ is chosen sufficiently large, depending only on~$(L,d,\Lambda)$. 

\smallskip

If we tentatively admit the assertion that~\eqref{e.double.induct.hyp} implies~\eqref{e.double.induct.wrap}, with the dependencies for the constants as explained above, we would then be able to iterate this to get, by induction, that, for a constant~$C(d,\Lambda)$,
\begin{align*}
N( C K^4 ( \log M)^4) \leq CM.
\end{align*}
We could then choose $L=C(d,\Lambda)$, which then removes this dependence from~$K$ to allow it to depend only on~$(d,\Lambda)$, and we would obtain
\begin{align}
\label{e.doubling.logM4.pf}
\sup_{r \in [ C(\log M)^4 , R]}
\frac{\left\| u \right\|_{\underline{L}^2(E_{r})}}{ \left\| u \right\|_{\underline{L}^2(E_{\theta r})} }
\leq 
CM\,.
\end{align}

\smallskip

We focus therefore on the proof that~\eqref{e.double.induct.hyp} implies~\eqref{e.double.induct.wrap}. Assuming $r_0$ and $L$ satisfy~\eqref{e.double.induct.hyp}, we
select~$m\in\N$ with~$m \geq K (\log M)$ and $C\theta^{-2} m^4 \leq r_0$. This choice of~$m$ is permitted by the lower bound on~$r_0$ in~\eqref{e.double.induct.hyp}.
We now apply large-scale analyticity (Proposition~\ref{p.analyticity})  to obtain~$\psi \in \A_m$ satisfying, for every~$r,s\in [Cm,R]$ with $r\leq s$, 
\begin{align}
\label{e.star}
\left\| u - \psi \right\|_{\underline{L}^2(E_r)}
\leq 
\biggl( \frac{Cr}{s} \biggr)^{\!\! m+1}
\left\| u  \right\|_{\underline{L}^2(E_s)}
\,.
\end{align}
By Proposition~\ref{p.threeball.Am}, for every $r \in  [r_0,r_0/\theta]$, 
\begin{align}
\label{e.apply.threeball.Am}
\frac
{\left\| \psi \right\|_{L^2(E_{\theta r})}^2}
{\left\| \psi \right\|_{L^2(  E_{\theta^2 r})}\left\| \psi \right\|_{L^2(E_{r})}}
\leq 
1 + \frac{Cm^4}{\theta^2 r}
\,.
\end{align}
We would like to replace each~$\psi$ on the left side of~\eqref{e.apply.threeball.Am} with~$u$, and this will be accomplished with the help of~\eqref{e.double.induct.hyp} and~\eqref{e.star}. 
Choose~$k\in\N$ such that $\theta^{k} R < r_0 \leq \theta^{k-1} R$ and, noting that $k\geq 3$ by~\eqref{e.double.induct.hyp}, apply~\eqref{e.double.ass} and~\eqref{e.double.induct.hyp} to obtain,
for every~$r \in [\theta r_0,r_0/\theta]$,
\begin{align}
\label{e.powers.and.logs}
\frac{\left\| u \right\|_{\underline{L}^2(E_{\theta^{2-k}r})}}{\left\| u \right\|_{\underline{L}^2(E_{r})}}
=
\prod_{j=1}^{k-2}
\frac{\left\| u \right\|_{\underline{L}^2(E_{\theta^{-j} r})}}{\left\| u \right\|_{\underline{L}^2(E_{\theta^{1-j}r})}}
\leq
(N(r_0))^{k-2}
&
\leq
(LM)^{ \log (R/r_0) / | \log \theta |}
\\ &  \notag
=
\biggl( \frac{R}{r_0} \biggr)^{\!\! (\log (LM)) / | \log \theta|} 
\,.
\end{align}
For any such~$r$, the right side of~\eqref{e.star} for~$s=\theta^{2-k} r \geq \theta^3 R$ may therefore be estimated by
\begin{align}
\label{e.Ktheta}
\biggl( \frac{Cr}{\theta^{-k} r} \biggr)^{\!\! m+1} \left\| u \right\|_{\underline{L}^2(E_{\theta^{2-k} r})}
& 
\leq
\biggl( \frac{Cr_0}{\theta^4 R} \biggr)^{\!\! m+1} 
\biggl( \frac{R}{r_0} \biggr)^{\!\! (\log (LM)) / | \log \theta|} 
\left\| u \right\|_{\underline{L}^2(E_{r})}
\\ & \notag
\leq
\biggl( \frac{C}{\theta^4} \biggr)^{\!\!m+1} 
\biggl( \frac{r_0}{R} \biggr)^{\!\! m - (\log (LM)) / | \log \theta|} 
\left\| u \right\|_{\underline{L}^2(E_{r})}
\,.
\\ & \notag
\leq 
\biggl( \biggl( \frac{C{r_0}}{R} \biggr)^{\!\! \sfrac m2} \wedge \frac12 \biggr)
\left\| u \right\|_{\underline{L}^2(E_{r})}
\,,
\end{align}
where the last line is obtained after first choosing~$K(\theta,L,d,\Lambda)$ large enough that $m\geq K \log M$ implies that $m - (\log (LM)) / | \log \theta| \geq \frac{9}{10} m$
and then selecting~$\theta(d,\Lambda)\in (0,\sfrac12]$ sufficiently small and using that~$r_0 \leq \theta^5 R$.

\smallskip

By the triangle inequality,~\eqref{e.star} and~\eqref{e.Ktheta}, 
we obtain, for every~$r \in  [r_0,r_0/\theta]$, 
\begin{align*}
\frac{\left\| \psi \right\|_{L^2(E_{ r})}}
{\left\| u \right\|_{L^2(E_{ r})}}
\leq
1 + \biggl( \frac{Cr_0}{R} \biggr)^{\!\! \sfrac m2} 
\quad \mbox{and} \quad 
\frac{\left\| u \right\|_{L^2(E_{\theta r})}}
{\left\| \psi \right\|_{L^2(E_{\theta r})}}
\leq
1 + \biggl( \frac{Cr_0}{R} \biggr)^{\!\! \sfrac m2} 
\,.
\end{align*}
For the smallest ellipsoid~$E_{\theta^2 r}$, we do not initially bound the ratio since, for~$r \in  [r_0,r_0/\theta]$, we have~\eqref{e.Ktheta} only with~$E_{\theta r}$ on the right side rather than~$E_{\theta^2r}$. So we must make do with the estimate
\begin{align*}
\left\| \psi \right\|_{L^2(E_{\theta^2 r})}
\leq
\left\| u \right\|_{L^2(E_{\theta^2 r})}
+
\biggl( \frac{Cr_0}{R} \biggr)^{\!\! \sfrac m2} 
\left\| u \right\|_{\underline{L}^2(E_{\theta r})}
\,.
\end{align*}
Combining these with~\eqref{e.apply.threeball.Am} and the triangle inequality yields, 
for~$r \in  [r_0,r_0/\theta]$, 
\begin{align*}
\lefteqn{
\left\| u \right\|_{L^2(E_{\theta r})}^2
} \qquad & 
\\ & 
\leq
\biggl(1 + \biggl( \frac{Cr_0}{R} \biggr)^{\!\! \sfrac m2} \, \biggl)^{\!\!2}
\left\| \psi \right\|_{L^2(E_{\theta r})}^2
\\ & 
\leq
\biggl(1 + \biggl( \frac{Cr_0}{R} \biggr)^{\!\! \sfrac m2} \!\! + \frac{Cm^4}{\theta^2 r} \biggl)
\left\| \psi \right\|_{L^2(  E_{\theta^2 r}) } \! 
\left\| \psi \right\|_{L^2(E_{r})}
\\ & 
\leq
\biggl(1 + \frac{Cm^4}{\theta^2 r} \biggl)
\left\| u \right\|_{L^2(E_{\theta^2 r})} \!
\left\| u \right\|_{\underline{L}^2(E_{r})}
+
\theta^{-\sfrac d2}
\biggl( \frac{Cr_0}{R} \biggr)^{\!\! \sfrac m2} \!\!
\left\| u \right\|_{\underline{L}^2(E_{\theta r})} \!
\left\| u \right\|_{\underline{L}^2(E_{r})}
\,.
\end{align*}
This implies that
\begin{align*}
\left\| u \right\|_{L^2(E_{\theta r})}
\leq
\biggl(1 + \frac{Cm^4}{\theta^2 r} \biggl)
\left\| u \right\|_{L^2(E_{\theta^2 r})}^{\sfrac 12} \!
\left\| u \right\|_{\underline{L}^2(E_{r})}^{\sfrac 12}
+
\theta^{-\sfrac d2}
\biggl( \frac{Cr_0}{R} \biggr)^{\!\! \sfrac m4} \!\!
\left\| u \right\|_{\underline{L}^2(E_{r})}
\,.
\end{align*}
By using~\eqref{e.double.induct.hyp} once more with~$r \geq r_0$, 
we may absorb the second term on the right side, provided that $Cr_0/R$ is sufficiently small; this can be assured if we shrink~$\theta$, since~$r_0\leq \theta^5 R$, and $K$ is chosen large compared to~$L$. 
Indeed:
\begin{align*}
\theta^{-\sfrac d2}
\biggl( \frac{Cr_0}{R} \biggr)^{\!\! \sfrac m4} \!\!
\left\| u \right\|_{\underline{L}^2(E_{r})}
\leq
LM
\theta^{-\sfrac d2}
( C\theta^5)^{\sfrac m4} 
\left\| u \right\|_{\underline{L}^2(E_{\theta r})}
\leq 
\frac 12
\left\| u \right\|_{\underline{L}^2(E_{\theta r})}
\,.
\end{align*}
We deduce therefore that
\begin{align}
\label{e.threesph}
\left\| u \right\|_{L^2(E_{\theta r})}
&
\leq
\biggl( 1 - \theta^{-\sfrac d2}
\biggl( \frac{Cr_0}{R} \biggr)^{\!\! \sfrac m4}  \biggr)^{\!\!-1}
\biggl(1 {+} \frac{Cm^4}{r} \biggr)^{\!\! \sfrac m4} 
\left\| u \right\|_{L^2(E_{\theta^2 r})}^{\sfrac 12} \!
\left\| u \right\|_{\underline{L}^2(E_{r})}^{\sfrac 12}
\\ & \notag
\leq
\biggl(1 + \frac{Cm^4}{r} + \biggl( \frac{Cr_0}{R} \biggr)^{\!\! \sfrac m4}\biggl)
\left\| u \right\|_{L^2(E_{\theta^2 r})}^{\sfrac 12} \!
\left\| u \right\|_{\underline{L}^2(E_{r})}^{\sfrac 12}
\\ &  \notag 
\leq
\biggl(1 + \frac{CK^4 (\log M)^4}{r} 
+ M^{-\frac18K \log \frac R{r_0}} \biggl)
\left\| u \right\|_{L^2(E_{\theta^2 r})}^{\sfrac 12} \!
\left\| u \right\|_{\underline{L}^2(E_{r})}^{\sfrac 12}
\,.
\end{align}
In particular, taking~$K$ sufficiently large, we obtain, for every $r \in [r_0, r_0/\theta]$, 
\begin{align*}
\frac{\left\| u \right\|_{\underline{L}^2(E_{\theta r})}}{ \left\| u \right\|_{\underline{L}^2(E_{\theta^2 r})} }
&
\leq 
\biggl(1 + \frac{CK^4 (\log M)^4}{r} + M^{-100 \log\frac R{r_0}} \biggl)
\frac{\left\| u \right\|_{\underline{L}^2(E_{r})}}{ \left\| u \right\|_{\underline{L}^2(E_{\theta r})} }
\\ & 
\leq
\biggl(1 + \frac{CK^4 (\log M)^4}{r} + M^{-100 \log \frac R{r_0}} \biggl) N(r_0)
\,.
\end{align*}
This completes the proof that~\eqref{e.double.induct.hyp} implies~\eqref{e.double.induct.wrap}, and hence also of the doubling estimate~\eqref{e.doubling.logM4.pf} down to the scale $C(\log M)^4$. 

\smallskip

Consequently, in view of~\eqref{e.threesph}, we also obtain that, for every $r \geq C (\log M)^4$,
\begin{align}
\label{e.threeball.logM4}
\left\| u \right\|_{L^2(E_{\theta r})} 
\leq 
\biggl(1+ \frac{C (\log M)^4}{R} + (CM)^{- 100 \log \frac{R}{r}} \biggr) 
\left\| u \right\|_{L^2(E_{\theta^2 r})}^{\sfrac12} 
\left\| u \right\|_{L^2(E_{r})}^{\sfrac12}
\,.
\end{align}
The three-ellipsoid and doubling inequalities have now been proved on scales larger than~$C(\log M)^4$. In fact, on these scales we have proved better estimates than what is stated in~\eqref{e.threeball.logM} and~\eqref{e.doubling.logM} as the double logarithm is absent. 

\smallskip

On scales between~$C\log M$ and $C(\log M)^4 \wedge R$, we argue differently: in fact, we use the large-scale analyticity estimate once more to crudely jump from scale~$C(\log M)^4$ all the way down to scale $C\log M$ in a single step. 
We therefore consider the case that $R \in [ C (\log M), \infty)$ for some large constant~$C$, define
\begin{align*}
R_0 = C (\log M)^4 \wedge R,
\end{align*}
and assume that the solution~$u$ satisfies the 
doubling inequality
\begin{align*}
\left\| u \right\|_{L^2(E_{R_0})}
\leq
M
\left\| u \right\|_{L^2(E_{\theta R_0})}
\,.
\end{align*}
Select $r\in [C \log M, R]$. Apply the large-scale analyticity estimate with $m= C \log M \leq cr \leq \theta R$ to obtain~$\psi \in \A_m$ such that
\begin{align}
\label{e.star.again}
\left\| u - \psi \right\|_{\underline{L}^2(E_{\theta R_0})}
\leq 
( C\theta)^{ m+1}
\left\| u  \right\|_{\underline{L}^2(E_{R_0})}
\,.
\end{align}
We deduce that 
\begin{align*}
\left\| u \right\|_{L^2(E_{\theta R_0})}
\leq 
\left\|  \psi \right\|_{\underline{L}^2(E_{\theta R_0})}
+
M ( C\theta )^{m+1}
\left\| u  \right\|_{\underline{L}^2(E_{R_0})}.
\end{align*}
Since $\theta=c$ and $m\geq C\log M$ with~$C$ as large as we like, we can absorb the second term on the right side; using then~\eqref{e.growth.psi}, we obtain
\begin{align*}
\left\| u \right\|_{L^2(E_{\theta R_0})}
\leq 
2\left\|  \psi \right\|_{\underline{L}^2(E_{\theta R_0})}
&
\leq
\biggl(\frac{CR_0}{r} \biggr)^{\!\!m}
\left\| u \right\|_{L^2(E_{r})}
=
M^{C \log \frac{R_0}{r} }
\left\| u \right\|_{L^2(E_{r})}
\,.
\end{align*}
Using the doubling assumption once more, we obtain
\begin{align*}
\frac{\left\| u \right\|_{L^2(E_{R_0})}}
{\left\| u \right\|_{L^2(E_{r})}}
\leq
M^{C \log \frac{R_0}{r} }
\,.
\end{align*}
Giving up volume factors, and using that 
\begin{align}
\label{e.notethat}
\log \tfrac{R_0}{r} 
\leq
(C\log \log M) \wedge \bigl( \log \tfrac{R}{r} \bigr)
\,,
\end{align}
we deduce that, for any $s_1, s_2 \in [r,R_0]$, 
\begin{align}
\label{e.wallop}
\frac{\left\| u \right\|_{L^2(E_{s_1})}}
{\left\| u \right\|_{L^2(E_{s_2})}}
\leq 
\biggl( \frac {R_0}r \biggr)^{\!\! d}
\frac{\left\| u \right\|_{L^2(E_{R})}}
{\left\| u \right\|_{L^2(E_{r})}}
&
\leq 
\biggl( \frac {R_0}r \biggr)^{\!\! d}
M^{C \log \frac{R_0}{r} }
\leq
M^{C \log \frac{R_0}{r} } 
\,.
\end{align}
The proof of the theorem is now complete. 
\end{proof}

\begin{remark}
\label{r.logM4}
Is the prefactor constant~$M^{C \log \log M}$ appearing on the right side of the estimates in Theorem~\ref{t.threeball} optimal? In the proof above we showed that, under the assumptions of Theorem~\ref{t.threeball}, on scales larger than $r\geq C(\log M)^4$,
\begin{align}
\label{e.threeball.logM4.2}
\left\| u \right\|_{L^2(E_{\theta r})} 
\leq 
\biggl(1+ \frac{C (\log M)^4}{r} + (CM)^{- 100 \log \frac{R}{r}} \biggr) 
\left\| u \right\|_{L^2(E_{\theta^2 r})}^{\sfrac12} 
\left\| u \right\|_{L^2(E_{r})}^{\sfrac12}
\end{align}
and
\begin{align}
\label{e.doubling.logM4.pf.2}
\sup_{r \in [ C(\log M)^4 , R]}
\frac{\left\| u \right\|_{\underline{L}^2(E_{r})}}{ \left\| u \right\|_{\underline{L}^2(E_{\theta r})} }
\leq 
M^C\,.
\end{align}
It is natural to wonder whether our estimate of~$C(\log M)^4$ for the minimal scale for such estimates can perhaps be improved to~$C\log M$. 
In other words, can Theorem~\ref{t.threeball} be improved by removing the double logarithms of~$M$? 

\smallskip

On the contrary, we conjecture that the optimal minimal scale for estimates like~\eqref{e.threeball.logM4.2} and~\eqref{e.doubling.logM4.pf.2} is~$C(\log M)^p$ for some $p\in (1,2)$. In particular, we expect that constant in Theorem~\ref{t.threeball}, with its $C\log\log M$ exponent, is sharp. 
This question seems to reduce to the question of the scale at which~$\mathscr{A}$-harmonic polynomials satisfy a doubling inequality, and resolving it one way or the other, at least by the methods in this paper, 
would require some nontrivial information about the higher-order homogenized tensors---for instance, whether they are generic or rather possess some unexpected special structure beyond what is currently known. 
\end{remark}

\footnotesize

\subsection*{Acknowledgments}
S.A.~was partially supported by NSF grant DMS-2000200. T.K.~was supported by the Academy of Finland and the European Research Council (ERC) under the European Union's Horizon 2020 research and innovation programme (grant agreement No 818437). C.S.~was partially supported by NSF grant DMS-2137909.

\bibliographystyle{abbrv}
\bibliography{doubling}
\end{document}